\documentclass[12pt]{article}
\usepackage{amssymb,amsthm,amsmath,latexsym}
\usepackage{url}
\newtheorem{thm}{Theorem}
\newtheorem{prop}[thm]{Proposition}
\newtheorem{lem}[thm]{Lemma}

\theoremstyle{remark}
\newtheorem{rem}[thm]{Remark}
\newcommand{\FF}{\mathbb{F}}
\newcommand{\ZZ}{\mathbb{Z}}
\newcommand{\RR}{\mathbb{R}}

\newcommand{\cB}{\mathcal{B}}

\DeclareMathOperator{\wt}{wt}

\DeclareMathOperator{\supp}{supp}

\begin{document}

\title{Some restrictions on weight enumerators of 
singly even self-dual codes II}

\author{
Masaaki Harada\thanks{
Corresponding author.
Research Center for Pure and Applied Mathematics,
Graduate School of Information Sciences,
Tohoku University, Sendai 980--8579, Japan.}
and 
Akihiro Munemasa\thanks{
Research Center for Pure and Applied Mathematics,
Graduate School of Information Sciences,
Tohoku University,
Sendai 980--8579, Japan.}
}

\maketitle

\begin{abstract}
In this note,
we give some restrictions on the number of vectors
of weight $d/2+1$ in the shadow of a singly even
self-dual $[n,n/2,d]$ code.
This eliminates some of the possible weight enumerators 
of singly even self-dual $[n,n/2,d]$ codes
for $(n,d)=(62,12)$, $(72,14)$, $(82,16)$, $(90,16)$ and $(100,18)$.
\end{abstract}



\section{Introduction}

Let $C$ be a singly even self-dual code and
let $C_0$ denote the
subcode of codewords having weight $\equiv0\pmod4$.
Then $C_0$ is a subcode of codimension $1$.
The {\em shadow} $S$ of $C$ is defined to be $C_0^\perp \setminus C$.
Shadows for self-dual codes were introduced by Conway and
Sloane~\cite{C-S}
in order to 
derive new upper bounds for the minimum weight of
singly even self-dual codes, and to
provide restrictions on the weight enumerators of 
singly even self-dual codes.
The largest possible minimum weights of
singly even self-dual codes of lengths $n \le 72$ were
given in~\cite[Table~I]{C-S}.
The work was extended to lengths $74 \le n \le 100$ in~\cite[Table~VI]{DGH}.
We denote by $d(n)$ the largest possible minimum weight given 
in~\cite[Table~I]{C-S} and \cite[Table~VI]{DGH}
throughout this note.
The possible weight enumerators of singly
even self-dual codes having minimum weight $d(n)$
were also given in~\cite{C-S}
for lengths $n \le 64$ and $n=72$ (see 
also~\cite{DGH} for length $72$), and the work was
extended to lengths up to $100$ in~\cite{DGH}.
It is a fundamental problem to find which weight enumerators
actually occur among the possible weight enumerators 
(see~\cite{C-S} and \cite{Huffman}).

Some restrictions on the number of vectors
of weight $d/2$ in the shadow of a singly even
self-dual $[n,n/2,d]$ code were given in~\cite{HM}.
Also, some restrictions on the number of vectors
of weight $d/2+1$ in the shadow of a singly even
self-dual $[n,n/2,d]$ code were given in~\cite{BHM}
for $n \equiv 0 \pmod 4$.
In this note,
we improve the result in~\cite{BHM} about
the restriction on the number of vectors
of weight $d/2+1$ in the shadow of a singly even
self-dual $[n,n/2,d]$ code for $n \equiv 0 \pmod 4$.
We also give a restriction on the number of vectors
of weight $d/2+1$ in the shadow of a singly even
self-dual $[n,n/2,d]$ code for $n \equiv 2 \pmod 4$.
These restrictions
eliminate some of the possible weight enumerators determined 
in~\cite{C-S} and \cite{DGH}
for the parameters
$(n,d)=(62,12)$, $(72,14)$, $(82,16)$, $(90,16)$ and $(100,18)$

\section{Preliminaries}

A (binary) $[n,k]$ {\em code} $C$ is a $k$-dimensional vector subspace
of $\FF_2^n$,
where $\FF_2$ denotes the finite field of order $2$.
All codes in this note are binary.
The parameter $n$ is called the {\em length} of $C$.
The {\em weight} $\wt(x)$ of a vector $x \in \FF_2^n$ is
the number of non-zero components of $x$.
A vector of $C$ is a {\em codeword} of $C$.
The minimum non-zero weight of all codewords in $C$ is called
the {\em minimum weight} $d(C)$ of $C$ and an $[n,k]$ code with minimum
weight $d$ is called an $[n,k,d]$ code.
The \textit{dual code} $C^{\perp}$ of a code
$C$ of length $n$ is defined as
$
C^{\perp}=
\{x \in \FF_2^n \mid x \cdot y = 0 \text{ for all } y \in C\},
$
where $x \cdot y$ is the standard inner product.
A code $C$ is called \textit{self-dual} if $C = C^{\perp}$.
A self-dual code $C$ is {\em doubly even} if all
codewords of $C$ have weight divisible by four, and 
{\em singly even} if there exists at least one codeword of 
weight $\equiv 2 \pmod 4$.
Rains~\cite{Rains} showed that
the minimum weight $d$ of a self-dual code $C$ of length $n$
is bounded by
$d  \le 4 \lfloor{\frac {n}{24}} \rfloor + 6$
if $n \equiv 22 \pmod {24}$,
$d  \le 4  \lfloor{\frac {n}{24}} \rfloor + 4$
otherwise.
In addition, if $n \equiv 0 \pmod{24}$ and $C$ is singly even,
then $d  \le 4  \lfloor{\frac {n}{24}} \rfloor + 2$.
A self-dual code meeting the bound is called  {\em extremal}.
Let $A_i$ and $B_i$ be the numbers of vectors of
weight $i$ in $C$ and $S$, respectively.
The weight enumerators of $C$ and $S$ are given by
$\sum_{i=0}^n A_i y^i$
and 
$\sum_{i=d(S)}^{n-d(S)} B_i y^i$, respectively,
where $d(S)$ denotes the minimum weight of $S$.

Let $C$ be a singly even
self-dual code of length $n$ and let $S$ be the
shadow of $C$.
Let $C_0$ denote the subcode of codewords having weight 
$\equiv0\pmod4$.
There are cosets $C_1,C_2,C_3$ of $C_0$ such that
$C_0^\perp = C_0 \cup C_1 \cup C_2 \cup C_3 $, where
$C = C_0  \cup C_2$ and $S = C_1 \cup C_3$.

\begin{lem}[Conway and Sloane~\cite{C-S}]\label{lem:C-S}
Let $x_1,y_1$ be vectors of $C_1$ and 
let $x_3$ be a vector of $C_3$.
Then
$x_1+y_1 \in C_0$,
$x_1+x_3 \in C_2$ and 
$\wt(x_1) \equiv \wt(x_3) \equiv \frac{n}{2} \pmod 4$.
\end{lem}

\begin{lem}[Brualdi and Pless~\cite{BP}]\label{lem:BP}
Let $x_1,y_1$ be vectors of $C_1$ and 
let $x_3$ be a vector of $C_3$.
\begin{enumerate}
\renewcommand{\labelenumi}{\rm \arabic{enumi})}
\item Suppose that $n \equiv 0 \pmod 4$.
Then $x_1 \cdot y_1=0$ and $x_1 \cdot x_3=1$.
\item Suppose that $n \equiv 2 \pmod 4$.
Then $x_1 \cdot y_1=1$ and $x_1 \cdot x_3=0$.
\end{enumerate}
\end{lem}

\section{$n \equiv 2 \pmod 4$ and $d(S)=\frac{d(C)}{2}+1$}
\label{sec:2mod4}

Recall that the Johnson graph $J(v,d)$ has the collection $X$ of all
$d$-subsets of $\{1,2,\dots,v\}$ as vertices, and two distinct
vertices are adjacent whenever they share $d-1$ elements in common.
Assume $v\geq2d$ and set
\[R_i=\{(x,y)\in X\times X\mid |x\cap y|=d-i\}.\]
Then $\{R_i\}_{i=0}^d$ is a partition of $X\times X$.
The following lemma is known as Delsarte's inequalities since
it is the basis of Delsarte's linear programming bound.
We refer the reader to \cite{D} for an explicit formula
for the second eigenmatrix $Q$ appearing in the lemma.

\begin{lem}[{\cite[Proposition~2.5.2]{bcn}}]\label{lem:D}
Let $Y$ be a subset of vertices of $J(v,d)$, and set
\[a_i=\frac{1}{|Y|}|(Y\times Y)\cap R_i|\quad(0\leq i\leq d).\]
If we denote by $Q=(q^{(v)}_j(i))$ the second eigenmatrix
of $J(v,d)$, then every entry of the vector 
$(a_0,\dots,a_d)Q$ is nonnegative.
\end{lem}


Suppose that $Y$ is a subset of vertices of $J(v,d)$
such that two distinct members intersect at exactly
one element. Then by Lemma~\ref{lem:D}, every entry of the vector
\[
(1,0,\ldots,0,0,|Y|-1,0)Q
\]
is nonnegative, i.e.,
\[
q^{(v)}_j(0)+(|Y|-1)q_j^{(v)}(d-1)\geq0
\quad(1\leq j\leq d).
\]
Thus, we obtain 
\begin{equation}\label{eq:bound1}
|Y|\leq M_{v,d},
\end{equation}
where 
\[
M_{v,d}=\min\{ 1-\frac{q^{(v)}_j(0)}{q^{(v)}_j(d-1)}
\mid 1\leq j\leq d\text{ and }q_j^{(v)}(d-1)<0\}.
\]
If we define 
\[
M_{v,d}=\begin{cases}
2&\text{if $v=2d-1$,}\\
1&\text{if $d\leq v\leq 2d-2$,}\\
0&\text{if $0\leq v\leq d-1$,}
\end{cases}
\]
then~\eqref{eq:bound1} also holds for all $v,d$.

Now, let $C$ be a singly even self-dual code of length $n$
and let $S$ be the shadow of $C$.
For the remainder of this section, we assume that 
\begin{equation}\label{eq:2mod4}
n \equiv 2 \pmod 4 \text{ and } d(S)=\frac{d(C)}{2}+1.
\end{equation}
By Lemma~\ref{lem:C-S}, 
$d(C) \equiv n-2 \pmod 8$, and hence $d(S)$ is odd.

For each of $i=1,3$, let $Y_i$ be
the set of supports of vectors of weight $d(S)$ in 
$C_i$, and let $S_i$ be the union of the members 
of $Y_i$. 
From Lemma~\ref{lem:BP} and~\eqref{eq:2mod4}, we have the following:
\begin{equation}\label{eq:s}
| x \cap y| = 
\begin{cases}
1 & \text{ if } x,y \in Y_i,\; x\ne y, \\
0 & \text{ if } x \in Y_1,\; y \in Y_3.
\end{cases}
\end{equation}
Then by \eqref{eq:bound1}, we have
\[
|Y_i|\leq M_{|S_i|,d(S)}.
\]
It follows from~\eqref{eq:s} that $S_1\cap S_3=\emptyset$.
Thus, we have
\begin{equation}\label{bd}
B_{d(S)}=|Y_1|+|Y_3|\leq
\max\{ M_{v,d(S)}+M_{n-v,d(S)} \mid 0\leq v \leq n/2 \}.
\end{equation}

For $42 \le n \le 98$ and $d(C) =d(n)$, 
the parameters $(n,d(C),d(S))$ satisfying Condition~\eqref{eq:2mod4}
are listed in Table~\ref{Tab:2mod4},
where the values $d(n)$ are also listed in the table.
For some lengths $n$, the existence of a singly even self-dual
code of length $n$ and minimum weight $d(n)$
is currently not known.  In this case, we consider 
the case $d(C)=d(n)-2$.
We calculated the upper bound~\eqref{bd},
where the results are listed in Table~\ref{Tab:2mod4}.
This calculation was done by 
the program written in {\sc Magma}~\cite{Magma},
where the program is listed in~\ref{appendix}.

\begin{table}[thb]
\caption{Parameters satisfying~\eqref{eq:2mod4}}
\label{Tab:2mod4}
\begin{center}
{\small
\begin{tabular}{c|c|c|c|c}
\noalign{\hrule height0.8pt}
$n$ & $d(n)$ & $d(C)$&$d(S)$ & $B_{d(S)}$\\
\hline
42 &  8 & 8 & 5 &$\le 42$ \\
62 & 12 &12 & 7 &$\le 48$ \\
70 & 14 &12 & 7 &$\le 52$ \\
82 & 16 &16 & 9 &$\le 74$ \\
90 & 16 &16 & 9 &$\le 76$ \\
98 & 18 &16 & 9 &$\le 78$ \\
\noalign{\hrule height0.8pt}
\end{tabular}
}
\end{center}
\end{table}

We discuss the possible weight enumerators for the case
$d(n)=d(C)$ in Table~\ref{Tab:2mod4}.
The possible weight enumerators $W_{42}$ and $S_{42}$
of an extremal singly even self-dual $[42,21,8]$ code
with $d(S) \ge 5$
and its shadow are as follows~\cite{C-S}:
\[
\begin{array}{lll}
W_{42} &=& 1 + (84+8\beta)y^8 + (1449-24\beta)y^{10} + \cdots,\\
S_{42} &=& 
\beta y^5+ (896-8\beta)y^9 + (48384+28\beta)y^{13}+ \cdots, \\
\end{array}
\]
respectively,
where $\beta$ is an integer.
It was shown in~\cite{BHM42} that $0 \le \beta \le 42$.
Table~\ref{Tab:2mod4} gives an alternative proof.

The possible weight enumerators $W_{62}$ and $S_{62}$
of an extremal singly even self-dual $[62,31,12]$ code
with $d(S)\ge 7$
and its shadow are as follows~\cite{C-S} (see also~\cite{DH}):
\[
\begin{array}{lll}
W_{62} &=& 1 + (1860+32\beta)y^{12} + (28055-160\beta)y^{14}+ \cdots, \\
S_{62} &=& \beta y^7 + (1116-12\beta)y^{11}
+(171368     + 66 \beta) y^{15} + \cdots,
\end{array}
\]
respectively,
where $\beta$ is an integer with $0 \le \beta \le 93$.
Table~\ref{Tab:2mod4} gives the following:

\begin{prop}
If there exists an extremal singly even self-dual $[62,31,12]$ code
with weight enumerator $W_{62}$,
then $0 \le \beta \le 48$. 
\end{prop}

It is known that there exists an extremal singly even self-dual
$[62,31,12]$ code with weight enumerator $W_{62}$
for $\beta=0,2,9,10,15,16$ (see~\cite{Y14}).

The possible weight enumerators $W_{82}$ and $S_{82}$
of an extremal singly even self-dual $[82,41,16]$ code
with $d(S)\ge 9$
and its shadow are as follows~\cite{DGH}:
\[
\begin{array}{lll}
W_{82} &=& 1 + (39524 + 128 \alpha) y^{16} + (556985 - 896 \alpha) y^{18} 
+ \cdots, \\
S_{82} &=& \alpha y^9 + (1640 - \alpha) y^{13} +
(281424+ 120 \alpha) y^{17} + \cdots, \\
\end{array}
\]
respectively,
where $\alpha$ is an integer with
$0 \leq \alpha \leq \lfloor \frac{556985}{896} \rfloor=621$.
Table~\ref{Tab:2mod4} gives the following:

\begin{prop}
If there exists an extremal singly even self-dual $[82,41,16]$ code
with weight enumerator $W_{82}$,
then $0 \le \alpha \le 74$. 
\end{prop}

It is unknown whether there exists an extremal singly even self-dual 
code for any of these cases.

The possible weight enumerators $W_{90}$ and $S_{90}$
of an extremal singly even self-dual $[90,45,16]$ code
with $d(S)\ge 9$
and its shadow are as follows~\cite{DGH}:
\[
\begin{array}{lll}
W_{90}  &=&  1 + (9180 + 8 \beta)y^{16} + (-512 \alpha - 24 \beta +
224360)y^{18} + \cdots,\\
S_{90} &=&  \alpha y^9 + (\beta -18 \alpha) y^{13} +(112320 + 153 \alpha -16
\beta)y^{17} + \cdots,
\end{array}
\]
respectively,
where $\alpha$ and $\beta$ are integers with
$0 \le \alpha \le \frac{1}{18} \beta \le 
\lfloor \frac{224360}{24} \rfloor =9348$.
Table~\ref{Tab:2mod4} gives the following:

\begin{prop}
If there exists an extremal singly even self-dual $[90,45,16]$ code
with weight enumerator $W_{90}$,
then $0 \le \alpha \le 76$. 
\end{prop}

It is unknown whether there exists an extremal singly even self-dual 
code for any of these cases.

\section{$n \equiv 0 \pmod 4$ and $d(S)=\frac{d(C)}{2}+1$}
\label{sec:0mod4}

Let $C$ be a singly even self-dual code of length $n$
and let $S$ be the shadow of $C$.
In this section, we write $d(C)=d$ and $d(S)=s$ for short,
and assume that 
\begin{equation}\label{eq:0mod4}
n \equiv 0 \pmod 4 \text{ and } s=\frac{d}{2}+1.
\end{equation}
By Lemma~\ref{lem:C-S}, 
$d \equiv n-2 \pmod 8$, and hence $s$ is even.

\begin{prop}[\cite{BHM}]\label{prop:BHM}
Suppose that $n \equiv 0 \pmod 4$ and $s=\frac{d}{2}+1$.
Let $B_{s}$ denote the number of vectors of weight $s$ in $S$.
\begin{enumerate}
\renewcommand{\labelenumi}{\rm (\roman{enumi})}
\item If $2n>(d+2)^2$, then
\[B_{s}\leq\frac{2n}{d+2}.\]
\item If $(d+2)^2\leq 4n\leq2(d+2)^2$, then
\[B_{s}\leq d+2,\quad B_{s}\neq d+1.\]
\item If $4n<(d+2)^2$, then
\[B_{s}\leq 2\frac{2n-d-2}{d}.\]
\end{enumerate}
\end{prop}

The above proposition was essentially established by 
showing $B_{s}\leq\max\{l_1,l_2\}$, where
\begin{align*}
l_1&=\frac{2n}{d+2},\\
l_2&=\min\left\{d+2,2\frac{2n-d-2}{d}\right\}.
\end{align*}
We recall part of the proof of Proposition~\ref{prop:BHM}
for later use.
Denote the set of all vectors in $C_i$ of weight $s$ by 
$\cB_i$ ($i = 1,3$). Denote by $v * w$ the entrywise product
of two vectors $v,w$.
If $v, w \in \cB_i$, then $\wt(v * w) = 0$ and 
hence these vectors have disjoint supports. This implies
\begin{equation}\label{l1}
|\cB_i|\leq l_1\quad(i=1,3).
\end{equation}
If $v\in \cB_1$ and $w \in \cB_3$, then $\wt(v * w) = 1$. 
Thus, if 
$\cB_1$ and $\cB_3$ are both nonempty, then 
\begin{equation}\label{l2}
|\cB_i|\leq s.
\end{equation}

Using the following lemmas,
we give an improvement of the upper bound by showing
$B_{s}\leq\max\{l'_1,l'_2\}$, where
\begin{align*}
l'_1&=
\begin{cases}
l_1&\text{if $n$ is divisible by $2s$,}\\
2\left\lceil\frac{n-d+2}{d+2}\right\rceil-1
&\text{otherwise,}
\end{cases}\\
l'_2&=\begin{cases}
d+2-\left\lceil\sqrt{(d+2)^2-4n}\,\right\rceil
&\text{if $4n<(d+2)^2$,}\\
\min\left\{d+2,
4\left\lceil\frac{n-d+2}{d+2}\right\rceil-2\right\}
&\text{otherwise.}
\end{cases}
\end{align*}
Since
\begin{equation}\label{n2s}
\left\lceil\frac{n-d+2}{d+2}\right\rceil
=\left\lceil\frac{n/4-(s/2-1)}{s/2}\right\rceil
\leq
\frac{n}{2s},
\end{equation}
we have
\begin{equation}\label{l1'}
l'_1\leq l_1,
\end{equation}
and
\[4\left\lceil\frac{n-d+2}{d+2}\right\rceil-2
\leq
\frac{2n}{s}-2
<2\frac{2n-d-2}{d}.\]
The latter implies $l'_2\leq l_2$ provided $4n\geq(d+2)^2$.
If $4n<(d+2)^2$, then
\begin{align*}
&2\frac{2n-d-2}{d}-\left(d+2-\sqrt{(d+2)^2-4n}\right)
\\&=
\frac{\sqrt{(d+2)^2-4n}}{d}(d-\sqrt{(d+2)^2-4n})
\\&\geq0.
\end{align*}
Thus $l'_2\leq l_2$ holds in this case as well.
Therefore, the bound $B_s\leq\max\{l'_1,l'_2\}$ which will be
shown in Proposition~\ref{prop:imp} below is an improvement
of the bound given in Proposition~\ref{prop:BHM}.

\begin{lem}\label{lem:1}
Let
\[k=\left\lceil\frac{n-d+2}{2s}\right\rceil.\]
If $n$ is not divisible by $2s$, then
$|\cB_i|\leq2k-1$ for $i=1,3$.
\end{lem}
\begin{proof}
Suppose, to the contrary, $|\cB_i| \ge 2k$.
Then the sum of the all-one vector and the
$2k$ vectors of weight $s$
belongs to $C_0$ and has weight 
$n-2ks\leq d-2$.
This forces $n-2ks=0$, contradicting the assumption.
\end{proof}

\begin{lem}\label{lem:maxab}
Let $n$ and $s$ be positive integers with $n<s^2$. Then
\[\max\{a+b\mid a,b\in\ZZ,\;0\leq a,b\leq s,\;s(a+b)-ab\leq n\}
=2s-\left\lceil2\sqrt{s^2-n}\,\right\rceil.\]
\end{lem}
\begin{proof}
Since $n<s^2$, we have
\begin{align*}
&\max\{a+b\mid a,b\in\RR,\;0\leq a,b\leq s,\;s(a+b)-ab\leq n\}
\\&=\max\{a+b\mid 0\leq a,b\leq s,\;(s-a)b\leq n-sa\}
\\&=\max\{a+\min\{(n-sa)/(s-a),s\}\mid 0\leq a< s\}
\\&=\max\{(n-a^2)/(s-a)\mid 0\leq a< s\}.
\end{align*}
The function $f(x)=(n-x^2)/(s-x)$ defined on the interval $[0,s)$
has maximum $f(\alpha)=2\alpha$, where $\alpha=s-\sqrt{s^2-n}$.
Thus, we have
\begin{align*}
&\max\{a+b\mid a,b\in\ZZ,\;0\leq a,b\leq s,\;s(a+b)-ab\leq n\}
\\&\leq
\left\lfloor
\max\{a+b\mid a,b\in\RR,\;0\leq a,b\leq s,\;s(a+b)-ab\leq n\}
\right\rfloor
\\&=
\left\lfloor2\alpha\right\rfloor.
\end{align*}
Define $a,b\in\ZZ$ by
$a=\left\lfloor\alpha\right\rfloor$ and
\[b=\begin{cases}
\left\lfloor\alpha\right\rfloor&\text{if $\alpha-
\left\lfloor\alpha\right\rfloor<\frac12$,}\\
\left\lfloor\alpha\right\rfloor+1&\text{otherwise.}
\end{cases}\]
Then $a+b=\left\lfloor2\alpha\right\rfloor=
2s-\left\lceil2\sqrt{s^2-n}\,\right\rceil$. Since
$\alpha<s$, we have $b\leq s$. It remains to show
$s(a+b)-ab\leq n$, or equivalently,
\begin{equation}\label{eq:sab}
ab-s(a+b)+n\geq0.
\end{equation}
Observe
\[s-\left\lfloor\alpha\right\rfloor=\left\lceil\sqrt{s^2-n}\right\rceil.\]

If $\alpha-\left\lfloor\alpha\right\rfloor<\frac12$, then
\begin{align*}
ab-s(a+b)+n&=
\left\lfloor\alpha\right\rfloor^2-2s\left\lfloor\alpha\right\rfloor+n
\\&=
(s-\left\lfloor\alpha\right\rfloor)^2-(s^2-n)
\\&=
\left\lceil\sqrt{s^2-n}\right\rceil^2-(s^2-n)
\\&\geq0.
\end{align*}
Thus, \eqref{eq:sab} holds.

If $\alpha-\left\lfloor\alpha\right\rfloor\geq\frac12$, then
\[s-\left\lfloor\alpha\right\rfloor\geq\sqrt{s^2-n}+\frac12.\]
Thus
\begin{align*}
ab-s(a+b)+n&=
\left\lfloor\alpha\right\rfloor
(\left\lfloor\alpha\right\rfloor+1)
-s(2\left\lfloor\alpha\right\rfloor+1)+n
\\&=
(\left\lfloor\alpha\right\rfloor-s)
(\left\lfloor\alpha\right\rfloor+1-s)-(s^2-n)
\\&\geq
\left(\sqrt{s^2-n}+\frac12\right)
\left(\sqrt{s^2-n}-\frac12\right)-(s^2-n)
\\&=
-\frac14.
\end{align*}
Since $ab-s(a+b)+n$ is an integer, \eqref{eq:sab} holds.
\end{proof}

\begin{prop}\label{prop:imp}
Suppose that $n \equiv 0 \pmod 4$ and $s=\frac{d}{2}+1$.
Let $B_{s}$ denote the number of vectors of weight $s$
in $S$. Then
\begin{equation}\label{bd12}
B_{s}\leq\max\{l'_1,l'_2\}.
\end{equation}
More precisely,
\begin{enumerate}
\renewcommand{\labelenumi}{\rm (\roman{enumi})}
\item If $2n>d^2+6d$, then
\[B_{s}\leq
\begin{cases}
\frac{2n}{d+2}&\text{if $n$ is divisible by $2s$,}\\
2\left\lceil\frac{n-d+2}{d+2}\right\rceil-1
&\text{otherwise.}
\end{cases}
\]
\item If $(d+2)^2<2n\leq d^2+6d$, then
\[B_{s}\leq 
\begin{cases}
\frac{2n}{d+2}&\text{if $n$ is divisible by $2s$,}\\
d+2&\text{otherwise.}
\end{cases}\]
\item If $d^2+8d-4<4n\leq2(d+2)^2$, then
\[B_{s}\leq d+2,\quad B_{s}\neq d+1.\]
\item If $(d+2)^2\leq 4n\leq d^2+8d-4$, then
\[B_{s}\leq 4\left\lceil\frac{n-d+2}{d+2}\right\rceil-2.\]
\item If $4n<(d+2)^2$, then
\[B_{s}\leq d+2-\left\lceil\sqrt{(d+2)^2-4n}\,\right\rceil.\]
\end{enumerate}
\end{prop}
\begin{proof}
If one of $\cB_1$ and $\cB_3$ is empty, then 
\eqref{l1} and Lemma~\ref{lem:1} imply $B_{s}\leq l'_1$.
If 
$\cB_1$ and $\cB_3$ are both nonempty, then by \eqref{l2},
we have $B_{s}\leq2s=d+2$.
Moreover, suppose $n<s^2$. 
Observe
\[\left|\bigcup_{x\in\cB_1\cup\cB_3}\supp(x)\right|=
s(|\cB_1|+|\cB_3|)-|\cB_1||\cB_3|,\]
and this is at most $n$. By \eqref{l2}, we can apply
Lemma~\ref{lem:maxab} to conclude
\[B_{s}\leq2s-\left\lceil2\sqrt{s^2-n}\,\right\rceil.\]
Thus $B_{s}\leq l'_2$.
Therefore, \eqref{bd12} holds.

Next, we determine $\max\{l'_1,l'_2\}$.
If $2n>d^2+6d$, then
\[\frac{n-d+2}{d+2}>\frac12(d+2)\in\ZZ,\]
so
\begin{align*}
l'_1&\geq
2\left\lceil\frac{n-d+2}{d+2}\right\rceil-1
&&\text{(by \eqref{n2s})}
\\&\geq
2\left(\frac12(d+2)+1\right)-1
\\&=d+3
\\&\geq l'_2.
\end{align*}
Thus $\max\{l'_1,l'_2\}=l'_1$, and (i) holds.

Next suppose $(d+2)^2<2n\leq d^2+6d$.
Since
\begin{align*}
4\left\lceil\frac{n-d+2}{d+2}\right\rceil-2-(d+2)
&\geq
4\frac{n-d+2}{d+2}-2-(d+2)
\\&>\frac{d^2-2d+8}{d+2}
\\&>0,
\end{align*}
we have $l'_2=d+2$.
Since
\[\frac{n-d+2}{d+2}\leq \frac12(d+2)\in\ZZ,\]
we have
\[2\left\lceil\frac{n-d+2}{d+2}\right\rceil-1<d+2<l_1.\]
These imply
\[\max\{l'_1,l'_2\}=\begin{cases}
l_1&\text{if $n$ is divisible by $2s$,}\\
l'_2&\text{otherwise,}
\end{cases}\]
and (ii) holds.

Next suppose $(d+2)^2\leq 4n\leq2(d+2)^2$. 
We claim
\[l'_2=\begin{cases}
d+2&\text{if $4n\leq d^2+8d-4$,}\\
4\left\lceil\frac{n-d+2}{d+2}\right\rceil-2
&\text{otherwise.}
\end{cases}\]
Indeed, since $(d+4)/4=(s+1)/2\notin\ZZ$, we have
\begin{align*}
d+2>4\left\lceil\frac{n-d+2}{d+2}\right\rceil-2
&\iff
\frac{s}{2}\geq\left\lceil\frac{n-d+2}{d+2}\right\rceil
\\&\iff
\frac{s}{2}\geq\frac{n-d+2}{d+2}
\\&\iff
4n\leq d^2+8d-4.
\end{align*}
Since $4n\geq(d+2)^2$ and $d\neq4$, we have $n\geq 3d-2$. Thus
\[4\left\lceil\frac{n-d+2}{d+2}\right\rceil-2\geq\frac{2n}{d+2}.\]
This, together with $2n\leq(d+2)^2$ implies $l_1\leq l'_2$. Therefore,
$\max\{l'_1,l'_2\}=l'_2$. Now (iii) and (iv) hold
by Proposition~\ref{prop:BHM}~(ii).

Finally, suppose $4n<(d+2)^2$. Then it is easy to verify
\[\frac{2n}{d+2}\leq d+2-\sqrt{(d+2)^2-4n},\]
hence $\max\{l'_1,l'_2\}=l'_2$ by \eqref{l1'}. Thus (v) holds.
\end{proof}

\begin{rem}
In Proposition~\ref{prop:imp}~(v), it is sometimes possible to
draw a stronger conclusion
\[|\cB_i|=\frac12
\left(d+2-\left\lceil\sqrt{(d+2)^2-4n}\,\right\rceil\right)
\quad(i=1,3).\]
This is when a pair $\{a,b\}$ achieving the maximum in 
Lemma~\ref{lem:maxab} is unique. For the parameters
$(n,d,s)=(128,22,12)$, 
we necessarily have $|\cB_i|=8$ for $i=1,3$.
In general, a pair $\{a,b\}$ achieving the maximum in 
Lemma~\ref{lem:maxab} is not unique. For example, 
when $(n,d,s)=(120,22,12)$, both
$\{6,8\}$ and $\{7,7\}$ achieve the maximum.
\end{rem}

%

For only the parameters
$(n,d,s)=
(72, 14, 8)$ and 
$(100,18,10)$,
Proposition~\ref{prop:imp} gives an improvement over
Proposition~\ref{prop:BHM},
for $44 \le n \le 100$ and $d=d(n)$.
The bounds on $B_{s}$ obtained by Proposition~\ref{prop:imp}
are listed in Table~\ref{Tab:0mod4} for these parameters,
together with the part of Proposition~\ref{prop:imp} used,
where the bounds by Proposition~\ref{prop:BHM} are listed
in the last column.
The values $d(n)$ are also listed in the table.

\begin{table}[thb]
\caption{Parameters satisfying~\eqref{eq:0mod4}}
\label{Tab:0mod4}
\begin{center}
{\small
\begin{tabular}{c|c|c|c|c|c}
\noalign{\hrule height0.8pt}
$n$ & $d(n)$ & $d$&$s$ & Proposition~\ref{prop:imp} & Proposition~\ref{prop:BHM}\\
\hline
 72 & 14 & 14 & 8 & $B_{s}\le 14$ (iv)& $B_{s}\le 16$, $\ne 15$ \\
100 & 18 & 18 & 10 & $B_{s}\le 18$ (iv)& $B_{s}\le 20$, $\ne 19$ \\
\hline
108 & - & 18 &10 &$B_{s}\le 18$  (iv)& $B_{s}\le 20$, $\ne 19$ \\
116 & - & 18 &10 &$B_{s}\le 18$  (iv)& $B_{s}\le 20$, $\ne 19$ \\
128 & - & 22 &12 &$B_{s}\le 16$  (v)& $B_{s}\le 21$ \\
\noalign{\hrule height0.8pt}
\end{tabular}
}
\end{center}
\end{table}


We discuss the possible weight enumerators for the case
$d(n)=d$ in Table~\ref{Tab:0mod4}.
The possible weight enumerators of an extremal singly even
self-dual $[72,36,14]$ code with $s\ge 8$ and the shadow are as follows:
\[
\begin{array}{lll}
W_{72} &=& 1 + (8640 - 64 \alpha)y^{14} + (124281 + 384 \alpha) y^{16} 
+ \cdots, \\
S_{72} &=& 
\alpha y^8 + (546 - 14 \alpha) y^{12}  + (244584+91 \alpha)y^{16}
+ \cdots,
\end{array}
\]
respectively, 
where $\alpha$ is an integer with 
$0\leq \alpha\leq \lfloor \frac{546}{14} \rfloor=39$~\cite{DGH}.
We remark that Conway and Sloane~\cite{C-S} give only
two weight enumerators as the possible weight enumerators
of an extremal singly even self-dual $[72,36,14]$ code with 
$s\ge 8$ without reason, namely $\alpha=0,1$ in 
$W_{72}$.
Table~\ref{Tab:0mod4} shows the following:

\begin{prop}
If there exists an extremal singly even
self-dual $[72,36,14]$ code with weight enumerator $W_{72}$,
then $0 \le \alpha \le 14$.
\end{prop}

It is unknown whether there exists an extremal singly even self-dual 
code for any of these cases.

The possible weight enumerators of a singly even
self-dual $[100,50,18]$ code with $s\ge 10$ and the shadow are as follows:
\[
\begin{array}{lll}
W_{100} &=& 1 + (16 \beta + 52250)y^{18} 
+ (1024 \alpha - 64 \beta +972180)y^{20}
+ \cdots, \\
S_{100} &=& 
\alpha y^{10} + (-20 \alpha - \beta) y^{14} + (190 \alpha + 104500 +
18 \beta)y^{18} 
+ \cdots,
\end{array}
\]
respectively, 
where $\alpha,\beta$ are integers with 
$0 \leq \alpha \leq \frac{-1}{20} \beta \leq \frac{5225}{32}$~\cite{DGH}.
Table~\ref{Tab:0mod4} shows the following:

\begin{prop}
If there exists a singly even
self-dual $[100,50,18]$ code with weight enumerator $W_{100}$,
then $0 \le \alpha \le 18$.
\end{prop}

It is unknown whether there exists a singly even self-dual $[100,50,18]$ 
code for any of these cases.

We give more sets of parameters for which
the bound on $B_{s}$ obtained by 
Proposition~\ref{prop:imp}
improves the bound obtained by Proposition~\ref{prop:BHM}:
\[
(n,d,s)=
(108,18,10),
(116,18,10),
(128,22,12).
\]
These bounds are also listed in Table~\ref{Tab:0mod4}.

\bigskip
\noindent
{\bf Acknowledgment.}
This work was supported by JSPS KAKENHI Grant Number 15H03633.


\appendix
\def\thesection{Appendix \Alph{section}}
\section{}\label{appendix}
\begin{verbatim}
HahnPolynomial:=function(v,k,l,x)
  return (Binomial(v,l)-Binomial(v,l-1))*
      &+[ (-1)^i*Binomial(l,i)*Binomial(v+1-l,i)*
          Binomial(k,i)^(-1)*Binomial(v-k,i)^(-1)*
          Binomial(x,i) : i in [0..l] ];
end function;
Qmatrix:=function(v,k)
  return Matrix(Rationals(),k+1,k+1,
  [[HahnPolynomial(v,k,l,x) : l in [0..k] ]: x in [0..k]]);
end function;
boundM:=function(v,ds)
  if v le ds-1 then 
    return 0;
  elif v le ds*2-2 then
    return 1;
  elif v eq ds*2-1 then
    return 2;
  else
  Q:=Qmatrix(v,ds);
  return Min( { 1-Q[1][i+1]/Q[ds][i+1] : i in [0..ds] 
         | Q[ds][i+1] lt 0 } );
  end if;
end function;
res:=function(n,ds)
  bounds:=[ Floor(boundM(v,ds)+boundM(n-v,ds)): 
          v in {0..(n div 2)} ];
  max:=Max(bounds);
  return max;
end function;

X:=[[42,5],[62,7],[70,7],[82,9],[90,9],[98,9]];
[res(x[1],x[2]): x in X] eq [42,48,52,74,76,78];
\end{verbatim}
\end{document}